\documentclass[10pt]{article}

\usepackage[rgb,dvipsnames]{xcolor}
\usepackage{stmaryrd}
\usepackage{amsfonts}
\usepackage{amssymb}
\usepackage{amsmath,amsthm}
\usepackage{latexsym}
\usepackage{amscd}
\usepackage{esint}
\usepackage{mathrsfs}
\usepackage{dsfont}

\usepackage{enumitem}
 
\begingroup 
 \makeatletter
 \@for\theoremstyle:=definition,remark,plain\do{%
 \expandafter\g@addto@macro\csname th@\theoremstyle\endcsname{%
 \addtolength\thm@preskip\parskip
 }%
 }
 \endgroup
 
\usepackage[titletoc,toc,title]{appendix}
 
 \usepackage[colorlinks = true,
 linkcolor = blue,
 urlcolor = cyan,
 citecolor = red,
 anchorcolor = blue,
 pagebackref]{hyperref}
 
\newcounter{noteCounter}
\usepackage[alphabetic,backrefs]{amsrefs}
 
 \usepackage[parfill]{parskip}
 \usepackage{anysize}
\marginsize{1in}{1in}{1in}{1in}

\newtheorem{prop}{Proposition}
\newtheorem{thm}[prop]{Theorem}
\newtheorem{lemm}[prop]{Lemma}
\newtheorem{coro}[prop]{Corollary}

\newtheorem*{claim*}{Claim}

\theoremstyle{definition}
\newtheorem{defi}[prop]{Definition}
\newtheorem{rmk}[prop]{Remark}

\newcommand{\CC}{\mathbb{C}}

\newcommand{\HH}{\mathbb{H}}

\newcommand{\NN}{\mathbb{N}}
\newcommand{\PP}{\mathbb{P}}

\newcommand{\RR}{\mathbb{R}}

\newcommand{\cJ}{\mathcal J}

\DeclareMathOperator{\tr}{tr}

\DeclareMathOperator{\Sym}{Sym}

\DeclareMathOperator{\Id}{Id}

\DeclareMathOperator{\Endo}{End}

\DeclareMathOperator{\vol}{\mathrm{vol}}

\newcommand{\bangle}[1]{\big\langle #1 \big\rangle}

\newcommand{\hk}{\mathbin{\! \hbox{\vrule height0.3pt width5pt depth 0.2pt \vrule height5pt width0.4pt depth 0.2pt}}}

\newcommand{\rom}[1]{\expandafter\romannumeral #1}
\newcommand{\Rom}[1]{\uppercase\expandafter{\romannumeral #1}}

\setlist[enumerate]{leftmargin = 2em}
\allowdisplaybreaks

\numberwithin{equation}{section}

\begin{document}

\title{Holomorphicity of parabolic stable minimal surfaces \\ of high codimension}

\author{Da Rong Cheng \\ \textit{Department of Mathematics, University of Miami} \\ \tt{darong.cheng@miami.edu} \and Spiro Karigiannis \\ \textit{Department of Pure Mathematics, University of Waterloo} \\ \tt{karigiannis@uwaterloo.ca} \and Jesse Madnick \\ \textit{Department of Mathematics and Computer Science, Seton Hall University} \\ \tt{jesse.madnick@shu.edu}}

\maketitle

\begin{abstract}
A classical theorem of Micallef says that if $F \colon (\Sigma, g) \to \RR^4$ is a stable minimal immersion of an oriented $2$-dimensional complete Riemannian manifold (that is parabolic) into $\RR^4$, it is necessarily holomorphic with respect to some parallel orthogonal complex structure on $\RR^4$. We generalize this theorem by replacing $\RR^4$ with $\RR^{2 + 2k}$ for any codimension $2k$, under the additional hypothesis that the normal bundle $N \Sigma$ is equipped with a complex structure that is compatible with the induced metric and parallel with respect to the induced connection. This is a necessary assumption for such a theorem to hold (as it is a consequence of the conclusion), and it is automatically satisfied in the classical case $k=1$. We also briefly discuss possible further generalizations of such a result to other calibrations and to Smith maps.
\end{abstract}

\tableofcontents

\section{Introduction} \label{sec:intro}

A ubiquitous theme in geometry is that certain preferred classes of geometric objects arise as critical points of natural first order functionals, and thus satisfy the associated second order Euler-Lagrange equation. Some examples include: Yang-Mills connections, minimal immersions, and harmonic maps. In certain cases, if there is additional geometric structure present, there can exist \emph{special} solutions to such equations which actually satisfy a first order PDE that implies the second order equation, and moreover such special solutions are \emph{local minimizers}, rather than just critical points. (In fact, they are ``global'' minimizers in an appropriate sense.) Some examples include: instanton connections, calibrated immersions, and holomorphic maps between compact almost K\"ahler manifolds. 

It is thus natural to ask the reverse question: suppose that we had a \emph{stable} (locally minimizing) critical point of such a functional. Under what additional hypotheses can we guarantee that in fact such an object must be one of these special first-order solutions? There are numerous such theorems in the literature. Here we state only three classic examples:
\begin{itemize}
\item A stable Yang-Mills connection over $S^4$ with gauge group $\mathrm{SU}(2)$ or $\mathrm{SU}(3)$ must be an instanton (self-dual or anti-self-dual). This is a theorem of Bourguignon--Lawson~\cite{BL}.
\item A stable harmonic map from a compact Riemann surface into a compact K\"ahler manifold with positive holomorphic bisectional curvature must be $\pm$-holomorphic. This is a theorem of Siu--Yau~\cite{SY}. (See also~\cite{EL}.)
\item A stable minimal immersion of an oriented complete Riemannian $2$-dimensional manifold $(\Sigma, g)$ (that is parabolic) into $\RR^4$ must be holomorphic with respect to a parallel orthogonal complex structure on $\RR^4$. This is a theorem of Micallef~\cite[Theorem I]{Micallef}. (Here ``parabolic'' is a technical assumption on $(\Sigma, g)$ which says that every positive function $h$ on $\Sigma$ satisfying $\Delta h \leq 0$ must be constant.)
\end{itemize}

In~\cite{Micallef}, Micallef also generalized this last theorem to higher-dimensional ambient Euclidean spaces, assuming that the oriented complete stable minimal surface $(\Sigma, g)$ has finite total curvature and genus zero. (See~\cite[Theorem IV]{Micallef}.) This was recently extended to the genus one case by Fraser--Schoen~\cite[Theorem 1.1]{FS}, under the additional hypothesis that the minimal immersions obtained by passing to finite covers of $(\Sigma, g)$ are all stable, a condition known as covering stability. In both theorems, the proof relies on classification results for holomorphic vector bundles over the type of Riemann surfaces under consideration, and has a different flavor than that of~\cite[Theorem I]{Micallef}.

In the present paper we obtain another generalization of~\cite[Theorem I]{Micallef} to higher codimensions, under a set of assumptions that permits us to carry over the ideas behind the original proof. The precise statement is the following.

\begin{thm}\label{thm:main}
Let $(\Sigma, g)$ be an oriented $2$-dimensional complete Riemannian manifold that is parabolic. Given $k \geq 1$, suppose $F \colon (\Sigma, g) \to \RR^{2 + 2k}$ is a \emph{stable minimal} isometric immersion. Assume further that the normal bundle $N\Sigma$ admits a complex structure $J_N$ that is compatible with the induced metric and parallel with respect to the induced connection. Then $\Sigma$ is holomorphic with respect to some parallel orthogonal complex structure $\overline{J}$ on $\RR^{2 + 2k}$.
\end{thm}

Our assumption that $N\Sigma$ admits a parallel complex structure compatible with the induced metric can alternatively be phrased as saying that the induced connection $D$ on $N\Sigma$ has holonomy contained in a conjugate of $\mathrm{U}(k)$. In fact, Theorem~\ref{thm:main} can be equivalently reformulated as saying that if $\Sigma$ is complete, oriented, parabolic, and stable minimal, then $\Sigma$ is holomorphic with respect to some parallel orthogonal complex structure on $\RR^{2 + 2k}$ if and only if the holonomy of the normal connection is contained in a conjugate of $\mathrm{U}(k)$. 

Note that our assumption in Theorem \ref{thm:main} on $N \Sigma$ is necessary for such a theorem to hold, because the restriction of $\overline{J}$ to $N \Sigma$ is a complex structure $J_N$ on $N\Sigma$ which is compatible with the induced metric and parallel with respect to the induced connection. That is, our assumption is a consequence of the conclusion of Theorem~\ref{thm:main}. Moreover, this assumption is automatically satisfied in the classical case $k=1$ because $N\Sigma$ is an $\RR^2$-bundle over $\Sigma$ with structure group $\mathrm{SO}(2) = \mathrm{U}(1)$, so the complex structure $J_N$ on $N \Sigma$ is parallel because it depends only on the orientation.

We make a few remarks on the hypotheses of Theorem~\ref{thm:main} on the one hand, and those of~\cite[Theorem IV]{Micallef} and~\cite[Theorem 1.1]{FS} on the other. In all three theorems, the list of assumptions begins with a stable minimal immersion into $\RR^n$ of an oriented, complete $2$-dimensional Riemannian manifold $(\Sigma, g)$. To this we add the parabolicity condition as in the $4$-dimensional theorem of Micallef,~\cite[Theorem I]{Micallef}, and also assume that the codimension is even, and that $N\Sigma$ admits a complex structure with the stated properties. On the other hand, both~\cite[Theorem IV]{Micallef} and~\cite[Theorem 1.1]{FS} assume that $(\Sigma, g)$ has finite total curvature, which is known to imply parabolicity, by a classical result of Chern--Osserman~\cite{CO} and the fact that $(\Sigma, g)$ is complete and minimally immersed. In addition,~\cite[Theorem 1.1]{FS} strengthens the stability requirement to covering stability. That said, neither~\cite[Theorem IV]{Micallef} nor~\cite[Theorem 1.1]{FS} assume a priori that the codimension is even, but rather deduce as a result that $\Sigma$ lies in an even-dimensional subspace (on which there is a parallel orthogonal complex structure that makes it holomorphic).

We prove Theorem~\ref{thm:main} in Section~\ref{sec:main}. We discuss possible further generalizations to other calibrations and to Smith maps in Section~\ref{sec:conclusion}. Unless stated otherwise, all manifolds are assumed to be connected.

\textbf{Acknowledgements.} The research of SK is partially supported by a Discovery Grant from NSERC. The authors are grateful to the anonymous referee of an earlier version of this article, whose suggestions helped us to improve and clarify the exposition.

\section{Proof of the main theorem} \label{sec:main}

\subsection{Preliminaries} \label{sec:prelim}

Let $(\Sigma^{2m}, g)$ be a complete oriented Riemannian manifold such that every positive function $h \colon \Sigma \to (0, \infty)$ satisfying $\Delta h \leq 0$ is constant. (When $\Sigma$ is $2$-dimensional, this condition is known as \emph{parabolicity}. See the beginning of \cite[Section 5]{Micallef} for a number of conditions which imply parabolicity.)

We work in slightly more generality than needed for Theorem~\ref{thm:main}, only specializing to the case where $\Sigma$ is a surface when necessary. Specifically, we only need to suppose that $m=1$ in Lemmas~\ref{lemm:Apm-sym-tr} and~\ref{lemm:Q-surface} and throughout Section~\ref{sec:holomo}.

Next we let $F \colon (\Sigma^{2m}, g) \to \RR^{2m + 2k}$ be an isometric immersion. Then the pullback bundle $F^*(T\RR^{2m + 2k})$ splits orthogonally into $T\Sigma \oplus N\Sigma$. Accordingly, the Levi-Civita connection of the standard metric, denoted $\overline{\nabla}$, induces metric-compatible connections on $T\Sigma$ and $N\Sigma$, denoted by $\nabla$ and $D$, respectively. That is, 
\[
\nabla v = \overline{\nabla}^T v, \text{ for all }v \in \Gamma(T\Sigma),
\]
\[
D s = \overline{\nabla}^\perp s, \text{ for all }s \in \Gamma(N\Sigma).
\]
On $T\Sigma$, we assume that there is a $\nabla$-parallel complex structure $J_\Sigma$ compatible with $g$, thus making $\Sigma$ a K\"ahler manifold. (Note that this is automatic if $m=1$.)

Regarding $N\Sigma$, we assume that there exists an orthogonal almost complex structure $J_N \in \Gamma(\Endo N\Sigma)$ which is $D$-parallel on $N\Sigma$. More specifically, we require that $J_N$ satisfy
\begin{enumerate}
\item[(i)] $\bangle{J_N \xi, J_N\eta} = \bangle{\xi, \eta}$ for all $x \in \Sigma$ and $\xi, \eta\in N_x\Sigma$.
\vskip 1mm
\item[(ii)] $J_N^2 = -\Id$.
\vskip 1mm
\item[(iii)] $D(J_Ns) = J_N(D s)$ for all $s \in \Gamma(N\Sigma)$.
\end{enumerate}

The immersion $F$ gives rise to the second fundamental form $A$ of $\Sigma$, which is a section of $\Sym^2 (T^*\Sigma)\otimes N\Sigma$. Our convention for the second fundamental form is
\[
A(v, w) = \overline{\nabla}_v^\perp w.
\]
Also, given $v, w \in T_x\Sigma$ and $s \in N_x\Sigma$, we write $A^s_{v,w} = A^s(v, w)$ for $\bangle{A(v, w), s}$. Note then that
\begin{equation} \label{eq:2ndff}
A^s(v, w) = \bangle{\overline{\nabla}_v^\perp w, s} = - \bangle{\overline{\nabla}_v^T s, w}.
\end{equation}
Throughout this section we assume that $F$ is a \emph{stable minimal} isometric immersion, which means the following. Recall that $F$ is said to be a minimal immersion if
\[
\tr_\Sigma A = \sum_{i = 1}^{2m}A(\tau_i, \tau_i) = 0 \text{ for all }x \in \Sigma
\]
(here $\tau_1, \cdots, \tau_{2m}$ is an orthonormal basis for $T_x\Sigma$), which is equivalent to $F$ being a critical point of the volume functional with respect to compactly supported variations. To define stability, we recall further that if $F$ is a minimal immersion, then the second variation at $\Sigma$ of the volume is given by the formula
\begin{equation}\label{eq:2nd-variation}
\delta^2 V(s, s): = \int_{\Sigma} \big( |D s|^2 - |A^s|^2 \big) \vol_\Sigma,
\end{equation}
defined for compactly supported sections $s$ of $N\Sigma$. (Here $|A^s|^2 = \sum_{i, j = 1}^{2m} |A^s_{\tau_i, \tau_j}|^2$.) When $F$ is a minimal immersion and 
\[
\delta^2 V(s, s) \geq 0 \text{ for all compactly supported } s \in \Gamma(N\Sigma),
\] 
we say that $F$ is \emph{stable}.

We end this section with a version of~\eqref{eq:2nd-variation} that involves a cut off function, which roughly speaking allows us to substitute $s$ which are not necessarily compactly supported into~\eqref{eq:2nd-variation}. The proof, included for the convenience of the reader, is a standard computation. (See for example~\cite[equations (5.7) and (5.8)]{Micallef}.)

\begin{lemm}\label{lemm:2nd-variation-with-cutoff}
Let $s \in \Gamma(N\Sigma)$ and $f \in C^\infty_c(\Sigma)$. Then 
\[
\delta^2 V(fs, fs) = \int_{\Sigma} \big( |\nabla f|^2 |s|^2 + f^2 \bangle{\cJ(s), s} \big) \vol_\Sigma.
\]
Recall that $\cJ$ denotes the Jacobi operator, given by
\begin{equation} \label{eq:Jacobi}
\cJ(s) = D^* D s - A^s_{\tau_i, \tau_j} A_{\tau_i, \tau_j},
\end{equation}
where $D^* D =- D^2_{\tau_i, \tau_i}$.
\end{lemm}
\begin{proof}
We begin by noting that 
\[
\begin{split}
|D(fs)|^2 =\ & |(\nabla f)s + f D s|^2\\
=\ & |\nabla f|^2 |s|^2  + f^2 |D s|^2 + 2\bangle{(\nabla f)s, fD s}\\
=\ & |\nabla f|^2 |s|^2 + f^2 |D s|^2 + \frac{1}{2}\bangle{\nabla (f^2), \nabla |s|^2}.
\end{split}
\]
Integrating by parts over $\Sigma$ gives
\begin{equation*}
\begin{split}
\int_{\Sigma} |D(fs)|^2 \vol_{\Sigma} =\ & \int_{\Sigma} \big( |\nabla f|^2 |s|^2 + f^2 |D s|^2 + \frac{1}{2}\bangle{\nabla (f^2), \nabla |s|^2} \big) \vol_{\Sigma} \\
=\ & \int_{\Sigma} \big( |\nabla f|^2 |s|^2 + f^2 |D s|^2 - \frac{1}{2} f^2 \Delta |s|^2 \big) \vol_{\Sigma}.
\end{split}
\end{equation*}
Combining this with
\[
\Delta |s|^2 = 2|D s|^2 - 2\bangle{s, D^* D s},
\]
we find that 
\[
\int_{\Sigma} |D(fs)|^2 \vol_{\Sigma} = \int_{\Sigma} \big( |\nabla f|^2 |s|^2 + f^2 \bangle{s, D^* D s} \big) \vol_{\Sigma}.
\]
Therefore we have
\[
\begin{split}
\delta^2 V(fs, fs) =\ & \int_{\Sigma} \big( |D (fs)|^2 - |A^{fs}|^2 \big) \vol_{\Sigma}\\
=\ & \int_{\Sigma} \big( |\nabla f|^2 |s|^2 + f^2 \bangle{s, D^* D s} - f^2 |A^s|^2 \big) \vol_{\Sigma}\\
=\ & \int_{\Sigma} \Big( |\nabla f|^2 |s|^2 +  f^2 \big(  \bangle{s, D^* D s} - |A^s|^2 \big) \Big) \vol_{\Sigma}.
\end{split}
\]
Noting that 
\[
|A^s|^2 = A^s_{\tau_i, \tau_j}\bangle{A_{\tau_i, \tau_j}, s} = \bangle{s, A^s_{\tau_i, \tau_j}A_{\tau_i, \tau_j}}
\]
completes the proof.
\end{proof}

\subsection{Special variations}

Similar to~\cite[Theorem I]{Micallef}, in the course of proving Theorem~\ref{thm:main}, the variations $s$ which we substitute into the second variation formula in Lemma~\ref{lemm:2nd-variation-with-cutoff} belong to a particular family constructed out of parallel vector fields on $\RR^{2m + 2k}$.

\begin{lemm}\label{lemm:a-Jacobi}
Given $a \in \RR^{2m + 2k}$, denote by $a^\perp$ the section of $N\Sigma$ obtained by letting
\[
a^\perp|_x = \text{orthogonal projection of }a \text{ onto }N_x\Sigma.
\] 
Then we have
\[
\cJ(a^\perp) = 0.
\]
\end{lemm}
\begin{proof}
Given $x \in \Sigma$, let $\tau_1, \cdots, \tau_{2m}$ be a local orthonormal frame for $T\Sigma$ such that $\nabla\tau_i = 0$ at $x$. By the constancy of $a$ and the definition of the second fundamental form, we have
\begin{equation}\label{eq:a-first-derivative}
\begin{split}
D_{\tau_i} a^\perp =\ & \overline{\nabla}^\perp_{\tau_i} (a - a^T) = -\overline{\nabla}_{\tau_i}^\perp a^T = -A(\tau_i, a^T).
\end{split}
\end{equation}
Differentiating a second time yields, at $x$, that 
\begin{align} \nonumber 
D^2_{\tau_i, \tau_i} a^\perp =\ & D_{\tau_i} D_{\tau_i} a^\perp = -D_{\tau_i} (A(\tau_i, a^T)) \\ \label{eq:a-2nd-derivative}
=\ & - (D A)(\tau_i, \tau_i, a^T)  - A(\tau_i, \nabla_{\tau_i} a^T).
\end{align}
By the symmetry of the second fundamental form and the Codazzi equation (here we use that the ambient space is flat), we have
\begin{equation}\label{eq:Codazzi}
(D A)(\tau_i, \tau_i, a^T) = (DA)(\tau_i, a^T, \tau_i) = (D A)(a^T, \tau_i, \tau_i).
\end{equation}
At the point $x \in \Sigma$, using the minimality of $F$, we find that
\begin{equation}\label{eq:minimality-used}
(D A)(a^T, \tau_i, \tau_i) = D_{a^T}(A(\tau_i, \tau_i)) = 0.
\end{equation}
Turning to the second term of~\eqref{eq:a-2nd-derivative}, using the constancy of $a$ and~\eqref{eq:2ndff} we have
\[
\nabla_{\tau_i} a^T = \overline{\nabla}^T_{\tau_i} (a - a^\perp) = - \overline{\nabla}^T_{\tau_i} a^\perp = A^{a^\perp}(\tau_i).
\]
Using this and~\eqref{eq:Codazzi},~\eqref{eq:minimality-used} in~\eqref{eq:a-2nd-derivative}, we get
\[
\begin{split}
D^2_{\tau_i, \tau_i} a^\perp =\ & - A(\tau_i, A^{a^\perp}(\tau_i)) =  - A^{a^\perp}(\tau_i, \tau_j) A(\tau_i, \tau_j).
\end{split}
\]
Thus we conclude that
\begin{equation}\label{eq:a-laplace}
D^* D a^\perp = A^{a^\perp}(\tau_i, \tau_j)A(\tau_i, \tau_j),
\end{equation}
which by~\eqref{eq:Jacobi} means precisely that $\cJ(a^\perp) = 0$, as asserted.
\end{proof}

\begin{lemm}\label{lemm:Q-computation}
For $a \in \RR^{2m + 2k}$, with $a^\perp$ defined as in Lemma~\ref{lemm:a-Jacobi}, we have
\begin{equation}\label{eq:Q-formula}
q(a): = \bangle{\cJ(J_N a^\perp), J_Na^\perp} = |A^{a^\perp}|^2 - |A^{J_Na^\perp}|^2.
\end{equation}
Moreover, the trace of the quadratic form $q$ over $\RR^{2m + 2k}$ vanishes. That is, letting $e_1, \cdots, e_{2m + 2k}$ be any orthonormal basis of $\RR^{2m+2k}$, then
\begin{equation}\label{eq:Q-trace-free}
\sum_{i = 1}^{2m + 2k}q(e_i) = 0 \, \text{ on }\Sigma.
\end{equation}
\end{lemm}
\begin{proof}
By the assumption that $J_N$ is $D$-parallel and equation~\eqref{eq:a-laplace}, we have
\[
D^* D (J_N a^\perp) = J_N (D^* D a^\perp) = A^{a^\perp}_{ij}J_N A_{ij}.
\]
Therefore, 
\[
\cJ(J_Na^\perp) = A^{a^\perp}_{ij}J_NA_{ij} - A^{J_Na^\perp}_{ij}A_{ij},
\]
and hence using also the orthogonality of $J_N$ we obtain
\[
\begin{split}
\bangle{\cJ(J_N a^\perp), J_Na^\perp} =\ & A^{a^\perp}_{ij}  \bangle{J_N A_{ij}, J_N a^\perp} - A^{J_Na^\perp}_{ij}\bangle{A_{ij}, J_Na^\perp}\\
=\ & |A^{a^\perp}|^2 - |A^{J_N a^\perp}|^2.
\end{split}
\]
This proves~\eqref{eq:Q-formula}. To prove~\eqref{eq:Q-trace-free}, we fix $x \in \Sigma$ and note that the sum $\sum q(e_i)$ does not depend on the choice of orthonormal basis of $\RR^{2m + 2k}$, so we may assume that $e_{1}, \cdots ,e_{2m} \in T_x\Sigma$ and $e_{2m + 1}, \cdots, e_{2m + 2k} \in N_x\Sigma$. Then we have
\[
\sum_{i = 1}^{2m + 2k} q(e_i) = \sum_{j = 1}^{2k}q(e_{2m + j}) = \sum_{j = 1}^{2k} |A^{e_{2m+j}}|^2 - \sum_{j = 1}^{2k} |A^{J_Ne_{2m + j}}|^2.
\]
Since $J_N\big|_x \colon N_x \Sigma \to N_x\Sigma$ is orthogonal, we have
\[
\sum_{j = 1}^{2k} |A^{e_{2m+j}}|^2 = \sum_{j = 1}^{2k} |A^{J_Ne_{2m + j}}|^2,
\]
which gives~\eqref{eq:Q-trace-free}.
\end{proof}

\subsection{The infinitesimally holomorphic condition}\label{sec:inf-holo}

The quadratic form $q$ is closely related to the so-called infinitesimally holomorphic condition, which we now recall. For our purposes we need only take the ambient space to be $\RR^{2m + 2k}$, although the condition can be defined in more general settings. Since $F^*(T\RR^{2m + 2k}) = T\Sigma \oplus N\Sigma$, using the complex structures $J_\Sigma$ and $J_N$ on $T\Sigma$ and $N\Sigma$, respectively, we can build a complex structure $J$ on the vector bundle $F^*(T\RR^{2m + 2k})$ by defining 
\begin{equation} \label{eq:defn-J}
JX = J_\Sigma X^T + J_N X^\perp.
\end{equation}

\begin{prop}\label{prop:infinitesimally-holomorphic}
The following two conditions are equivalent:
\begin{enumerate}
\item[(a)] $\overline{\nabla} J = 0$, in the sense that $\overline{\nabla}(JX) = J(\overline{\nabla}X)$ for all $X \in \Gamma(T\Sigma \oplus N\Sigma)$.
\item[(b)] $A(J_\Sigma v, w) = J_N A(v, w)$, for all $x\in \Sigma$ and $v, w \in T_x \Sigma$.
\end{enumerate}
(The $m=k=1$ case of this is contained in~\cite[Proposition 3.4]{MW}.)
\end{prop}
\begin{proof}
First, let $v \in \Gamma(T\Sigma)$. Using the fact that $J_\Sigma$ is $\nabla$-parallel, we compute
\begin{equation}\label{eq:d-J-tangent}
\overline{\nabla}(Jv) = \overline{\nabla}(J_\Sigma v) = \overline{\nabla}^\perp(J_\Sigma v) + \nabla (J_\Sigma v) = A(J_\Sigma v, \cdot) + J_\Sigma \nabla v.
\end{equation}
We also have
\begin{equation}\label{eq:J-d-tangent}
J(\overline{\nabla} v) =  J_\Sigma (\overline{\nabla}^T v) + J_N(\overline{\nabla}^\perp v) = J_\Sigma \nabla v + J_N A(v, \cdot).
\end{equation}
Next, let $\xi \in \Gamma(N\Sigma)$. Similarly, using the fact that $J_N$ is $D$-parallel, we compute
\begin{equation}\label{eq:d-J-normal}
\overline{\nabla}(J\xi) = \overline{\nabla}(J_N \xi) = D(J_N\xi) + \overline{\nabla}^T (J_N \xi) = J_N D\xi - A^{J_N\xi}(\cdot).
\end{equation}
We also have
\begin{equation}\label{eq:J-d-normal}
J(\overline{\nabla} \xi) =  J_\Sigma (\overline{\nabla}^T \xi) + J_N(\overline{\nabla}^\perp \xi) = -J_\Sigma A^{\xi}(\cdot) + J_N D\xi.
\end{equation}
Comparing~\eqref{eq:d-J-tangent} and~\eqref{eq:J-d-tangent}, we see that (a) implies (b). To prove the converse, we first note that (b) implies, by~\eqref{eq:d-J-tangent} and~\eqref{eq:J-d-tangent}, that 
\[
\overline{\nabla}(Jv) = J(\overline{\nabla }v), \text{ for all }v \in \Gamma(T\Sigma)
\]
Next, we take the inner product of the condition in (b) with an arbitrary $\xi \in N_x\Sigma$ to see that (b) implies
\[
A^\xi(J_\Sigma v, w) = - A^{J_N\xi}(v, w).
\]
The left hand side can be rewritten as
\[
A^\xi(J_\Sigma v, w) = \bangle{J_\Sigma v, A^{\xi}(w)} = -\bangle{v, J_\Sigma A^{\xi}(w)},
\]
while for the right hand side we have
\[
- A^{J_N\xi}(v, w) = - \bangle{v, A^{J_N\xi}(w)}.
\]
Hence we obtain from the three above equations that $J_\Sigma A^{\xi}(\cdot) = A^{J_N\xi}(\cdot)$. From this and equations~\eqref{eq:d-J-normal} and~\eqref{eq:J-d-normal}, we deduce that
\[
\overline{\nabla}(J\xi) = J(\overline{\nabla}\xi), \text{ for all }\xi \in \Gamma(N\Sigma),
\]
completing the proof that (b) implies (a).
\end{proof}

\begin{defi}
Let $F \colon \Sigma^{2m} \to \RR^{2m+2k}$ be an isometric immersion, and suppose $T\Sigma$ and $N\Sigma$ are equipped with orthogonal complex structures which are parallel with respect to $\nabla$ and $D$, respectively. (In particular, this means that $(\Sigma, J_\Sigma)$ equipped with the induced metric is K\"ahler.) Define the orthogonal complex structure $J $ on $T\Sigma \oplus N\Sigma$ as in~\eqref{eq:defn-J}. We say that $F$ is \emph{$(J_\Sigma, J_N)$-infinitesimally holomorphic} if either of the two equivalent conditions of Proposition~\ref{prop:infinitesimally-holomorphic} hold. (When $m=k=1$, this condition is also sometimes called \emph{superminimal} or \emph{real isotropic}. See~\cite{Bryant}, for example.)
\end{defi}

\begin{rmk} \label{rmk:inf-hol-motivation}
Suppose that $\RR^{2m+2k}$ is equipped with a $\overline{\nabla}$-parallel orthogonal complex structure $\overline{J}$, and that the immersion $F \colon (\Sigma, J_\Sigma) \to (\RR^{2m+2k}, \overline{J})$ is holomorphic, in the sense that $F_* J_\Sigma = \overline{J} F_*$. Then, taking $J_N$ to be the restriction of $\overline{J}$ to $N \Sigma$, the second fundamental form $A$ of the immersion satisfies the $(J_\Sigma, J_N)$-infinitesimally holomorphic condition. That is, this is precisely the condition on $A$ which is necessary for there to exist a $\overline{\nabla}$-parallel orthogonal complex structure $\overline{J}$ on $\RR^{2m+2k}$ with respect to which the immersion is holomorphic.
\end{rmk}

\begin{rmk}\label{rmk:pm-holomorphic}
Another orthogonal complex structure on $F^* (T\RR^{2m + 2k}) = T\Sigma \oplus N\Sigma$ built from $J_\Sigma$ and $J_N$ is given by
\[
(J_-)X = J_\Sigma X^T - J_N X^\perp.
\]
There is an analogous characterization for $\overline{\nabla}J_- = 0$ in terms of the second fundamental form. That is, it is equivalent to the condition that
\[
A(J_\Sigma v, w) = - J_NA(v, w), \text{ for all }x \in \Sigma \text{ and }v, w \in T_x\Sigma.
\]
Recall that we \emph{chose} an orthogonal $D$-parallel complex structure $J_N$ on the normal bundle $N\Sigma$ as part of our data. The other orthogonal complex structure $J_-$ on $T\Sigma \oplus N\Sigma$ described here corresponds to replacing $J_N$ by $-J_N$, which is still orthogonal and $D$-parallel.
\end{rmk}

In order to relate the $(J_\Sigma, J_N)$-infinitesimally holomorphic condition with the quadratic form $q$ from~\eqref{eq:Q-formula}, we define 
\begin{equation} \label{eq:defn-Apm}
A^{\pm}(v, w) = \frac{1}{2}\big( A(v, w) \mp J_NA(J_\Sigma v, w) \big), \text{ for all }x \in \Sigma,\ v, w \in T_x\Sigma.
\end{equation}
It is easy to see that 
\begin{equation}\label{eq:Apm-intertwine}
A^\pm(J_\Sigma v, w) = \pm J_N A^{\pm}(v, w).
\end{equation}
The next lemma clarifies the relationship between $A^{\pm}$ and $q$.

\begin{lemm}\label{lemm:Q-Apm}
We have 
\begin{equation}\label{eq:Q-Apm-1}
q(a) = 4\sum_{i, j = 1}^{2m}\bangle{A^+_{\tau_i, \tau_j}, a^\perp}\bangle{A^-_{\tau_i, \tau_j}, a^\perp}.
\end{equation}
Moreover, if we choose an orthonormal basis for $T_x\Sigma$ of the form $v_{1}, J_\Sigma v_1, \cdots, v_m, J_\Sigma v_m$, then 
\begin{equation}\label{eq:Q-Apm-2}
\begin{split}
q(a) =\ & 4\Big[ \sum_{i, j = 1}^m\bangle{A^+_{v_i, v_j}, a^\perp}\bangle{A^-_{v_i, v_j}, a^\perp} - \sum_{i, j = 1}^m \bangle{A^+_{v_i, v_j}, J_Na^\perp}\bangle{A^-_{ v_i, v_j}, J_Na^\perp}\Big]\\
& + 4 \Big[ \sum_{i, j = 1}^m \bangle{A^+_{ v_i, J_\Sigma v_j}, a^\perp}\bangle{A^-_{ v_i, J_\Sigma v_j}, a^\perp} - \sum_{i, j = 1}^m \bangle{A^+_{v_i, J_\Sigma v_j}, J_Na^\perp}\bangle{A^-_{ v_i, J_\Sigma v_j}, J_Na^\perp} \Big].
\end{split}
\end{equation}
\end{lemm}
\begin{proof}
We begin by noting that, since $J_\Sigma$ is orthogonal, we have
\[
|A^{J_Na^\perp}|^2 = \sum_{i, j=1}^{2m}\bangle{A_{\tau_i, \tau_j}, J_Na^\perp}^2 = \sum_{i, j=1}^{2m}\bangle{J_NA_{ \tau_i, \tau_j}, a^\perp}^2 = \sum_{i, j=1}^{2m}\bangle{J_NA_{J_\Sigma \tau_i, \tau_j}, a^\perp}^2.
\]
Using this and~\eqref{eq:Q-formula}, we compute 
\[
\begin{split}
q(a) =\ & |A^{a^\perp}|^2 - |A^{J_Na^\perp}|^2\\
=\ & \sum_{i, j}\bangle{A_{\tau_i, \tau_j}, a^\perp}^2 - \sum_{i, j}\bangle{J_NA_{J_\Sigma \tau_i, \tau_j}, a^\perp}^2\\
=\  & \sum_{i, j} \bangle{A_{\tau_i, \tau_j} -J_NA_{J_\Sigma \tau_i, \tau_j}, a^\perp }\bangle{A_{\tau_i, \tau_j} +J_NA_{J_\Sigma \tau_i, \tau_j}, a^\perp },
\end{split}
\]
which proves~\eqref{eq:Q-Apm-1}. To deduce~\eqref{eq:Q-Apm-2} when the orthonormal basis of $T_x\Sigma$ has the form described in the statement, we first split the right hand side of~\eqref{eq:Q-Apm-1} to get
\begin{equation}\label{eq:Q-Apm-split}
\begin{split}
\frac{q(a)}{4} = \ & \sum_{i, j = 1}^m\bangle{A^+_{v_i, v_j}, a^\perp}\bangle{A^-_{v_i, v_j}, a^\perp} + \sum_{i, j = 1}^m \bangle{A^+_{J_\Sigma v_i, v_j}, a^\perp}\bangle{A^-_{J_\Sigma v_i, v_j}, a^\perp}\\
& + \sum_{i, j = 1}^m \bangle{A^+_{ v_i, J_\Sigma v_j}, a^\perp}\bangle{A^-_{v_i, J_\Sigma v_j}, a^\perp} + \sum_{i, j = 1}^m \bangle{A^+_{J_\Sigma v_i, J_\Sigma v_j}, a^\perp}\bangle{A^-_{J_\Sigma v_i, J_\Sigma v_j}, a^\perp}.
\end{split}
\end{equation}
Now we use~\eqref{eq:Apm-intertwine} to rewrite the second and fourth summations in~\eqref{eq:Q-Apm-split}:
\[
\begin{split}
 \sum_{i, j = 1}^m \bangle{A^+_{J_\Sigma v_i, v_j}, a^\perp}\bangle{A^-_{J_\Sigma v_i, v_j}, a^\perp} =\ & \sum_{i, j = 1}^m \bangle{J_NA^+_{v_i, v_j}, a^\perp}\bangle{-J_NA^-_{v_i, v_j}, a^\perp}\\
 =\ & -\sum_{i, j = 1}^m \bangle{A^+_{v_i, v_j}, J_Na^\perp}\bangle{A^-_{v_i, v_j}, J_Na^\perp},
 \end{split}
\]
\[
\begin{split}
\sum_{i, j = 1}^m \bangle{A^+_{J_\Sigma v_i, J_\Sigma v_j}, a^\perp}\bangle{A^-_{J_\Sigma v_i, J_\Sigma v_j}, a^\perp} =\ & \sum_{i, j = 1}^{m}\bangle{J_N A^+_{v_i, J_\Sigma v_j}, a^\perp}\bangle{-J_N A^-_{v_i, J_{\Sigma}v_j}, a^\perp}\\
=\ &-\sum_{i, j = 1}^m \bangle{A^+_{v_i, J_\Sigma v_j}, J_Na^\perp}\bangle{A^-_{v_i, J_\Sigma v_j}, J_Na^\perp}.
\end{split}
\]
Rewriting the second and fourth terms in~\eqref{eq:Q-Apm-split} as indicated above yields~\eqref{eq:Q-Apm-2}.
\end{proof}

In the two following lemmas we specialize to the case where $m = 1$, so that $\Sigma$ is a surface, and establish some additional properties of $A^{\pm}$ for use in the final section where we prove Theorem~\ref{thm:main}.

\begin{lemm}\label{lemm:Apm-sym-tr}
Suppose that $m = 1$, so that $\Sigma$ is a surface. Then $A^{\pm}$ are both symmetric and trace-free.
\end{lemm}
\begin{proof}
Fix $x \in \Sigma$ and choose an orthonormal basis $\tau_1, \tau_2$ for $T_x\Sigma$ so that 
\[
\tau_2 = J_\Sigma\tau_1.
\]
We may then think of $A$ as a $2\times 2$ matrix with entries in $N_x\Sigma$:
\[
\left(
\begin{array}{cc}
A(\tau_1, \tau_1) & A(\tau_1, \tau_2)\\
A(\tau_2, \tau_1) & A(\tau_2, \tau_2)
\end{array}
\right).
\]
Letting $\widetilde{A}(v, w) = A(J_\Sigma v, w)$ for $v, w \in T_x\Sigma$, we observe that
\[
\left(
\begin{array}{cc}
\widetilde{A}(\tau_1, \tau_1) & \widetilde{A}(\tau_1, \tau_2)\\
\widetilde{A}(\tau_2, \tau_1) & \widetilde{A}(\tau_2, \tau_2)
\end{array}
\right) = 
\left(
\begin{array}{cc}
A(\tau_2, \tau_1) & A(\tau_2, \tau_2)\\
-A(\tau_1, \tau_1) & -A(\tau_1, \tau_2)
\end{array}
\right).
\]
In particular, $\widetilde{A}$ is trace-free if and only if $A$ is symmetric, and $\widetilde{A}$ is symmetric if and only if $A$ is trace-free. As $A$ has both properties, so does $\widetilde{A}$, and hence by~\eqref{eq:defn-Apm} so does $A^{\pm}$. 
\end{proof}

Using Lemma~\ref{lemm:Apm-sym-tr}, we can further simplify~\eqref{eq:Q-Apm-2}.

\begin{lemm}\label{lemm:Q-surface}
Suppose $m = 1$, so that $\Sigma$ is a surface. Then for all $x \in \Sigma$ and unit vector $\tau \in T_x\Sigma$, we have 
\begin{equation}\label{eq:Q-simplified}
q(a) = 8\bangle{A^+_{\tau, \tau}, a^\perp}\bangle{A^-_{\tau, \tau}, a^\perp}  - 8 \bangle{A^+_{\tau, \tau}, J_Na^\perp}\bangle{A^-_{\tau, \tau}, J_Na^\perp}.
\end{equation}
\end{lemm}
\begin{proof}
Since $m = 1$, each of the four summations in~\eqref{eq:Q-Apm-2} now involves only one single term, and we have
\begin{equation}\label{eq:Q-surface-split}
\begin{split}
\frac{q(a)}{4} =\ & \bangle{A^+_{\tau, \tau}, a^\perp}\bangle{A^-_{\tau, \tau}, a^\perp}  -\bangle{A^+_{\tau, \tau}, J_N a^\perp} \bangle{A^-_{\tau, \tau}, J_N a^\perp}\\
&+ \bangle{A^+_{\tau, J_\Sigma\tau}, a^\perp}\bangle{A^-_{\tau, J_\Sigma\tau}, a^\perp} - \bangle{A^+_{\tau, J_{\Sigma}\tau}, J_Na^\perp}\bangle{A^-_{\tau, J_{\Sigma}\tau}, J_N a^\perp}.
\end{split}
\end{equation}
Using the symmetry of $A^{\pm}$ given by Lemma~\ref{lemm:Apm-sym-tr}, in the third and fourth terms we can switch the order of the arguments in $A^\pm$ and then rewrite the terms as in the proof of Lemma~\ref{lemm:Q-Apm}. Specifically, we have
\[
\begin{split}
\bangle{A^+_{\tau, J_\Sigma\tau}, a^\perp}\bangle{A^-_{\tau, J_\Sigma\tau}, a^\perp} =\ &\bangle{A^+_{J_\Sigma\tau, \tau}, a^\perp}\bangle{A^-_{J_\Sigma\tau, \tau}, a^\perp}\\
=\ & -\bangle{J_NA^+_{\tau, \tau}, a^\perp}\bangle{J_NA^-_{\tau, \tau}, a^\perp}\\
=\ & -\bangle{A^+_{\tau, \tau}, J_N a^\perp}\bangle{A^-_{\tau, \tau}, J_N a^\perp},
\end{split}
\]
and similarly
\[
\begin{split}
\bangle{A^+_{\tau, J_{\Sigma}\tau}, J_Na^\perp}\bangle{A^-_{\tau, J_{\Sigma}\tau}, J_N a^\perp} =\ & \bangle{A^+_{J_{\Sigma}\tau, \tau}, J_N a^\perp}\bangle{A^-_{J_{\Sigma}\tau, \tau}, J_Na^\perp}
\\
=\ & -\bangle{J_NA^+_{\tau, \tau}, J_Na^\perp}\bangle{J_NA^{-}_{\tau, \tau}, J_Na^\perp}\\
=\ & -\bangle{A^+_{\tau, \tau}, a^\perp}\bangle{A^-_{\tau, \tau}, a^\perp}.
\end{split}
\]
Substituting these back into~\eqref{eq:Q-surface-split} gives
\[
\frac{q(a)}{4} = 2\bangle{A^+_{\tau, \tau}, a^\perp}\bangle{A^-_{\tau, \tau}, a^\perp} - 2\bangle{A^+_{\tau, \tau}, J_N a^\perp} \bangle{A^-_{\tau, \tau}, J_N a^\perp}. \qedhere
\]
\end{proof}

\subsection{A vanishing result}

In this section we return to general dimensions and prove a vanishing result for the quadratic form $q$ on $\RR^{2m + 2k}$ defined in Lemma~\ref{lemm:Q-computation}. This is where we use the assumptions that $F$ is stable and that $\Sigma$ does not support any non-constant superharmonic functions which are positive (recall that this assumption was called \emph{parabolicity}). The arguments in this section are essentially the same as those found in the work of Micallef~\cite{Micallef}, more precisely the proof of~\cite[Theorem I in Section 5]{Micallef}, up to some minor simplification. In particular, it turns out that we do not need to first derive a complexified version of the second variation formula (see~\cite[equation (2.2)]{Micallef}).

We will need to twice use the following celebrated result of Fischer-Colbrie and Schoen, which we state here in an equivalent formulation suitable for our purposes.

\begin{thm}[{Fischer-Colbrie--Schoen~\cite[Theorem 1, implication (i) $\Rightarrow$ (iii)]{S-FC}}] \label{thm:S-FC}
Let $p$ be a smooth function on $\Sigma$. Suppose that for all $f \in C^\infty_c(\Sigma)$ we have
\[
\int_{\Sigma} \big( |\nabla f|^2 + f^2 p \big) \vol_\Sigma \geq 0.
\]
Then there exists a \emph{positive} function $h$ on $\Sigma$ such that $\Delta h = p h$.
\end{thm}

\begin{prop}\label{prop:Q-vanish}
With the quadratic form $q$ as defined in Lemma~\ref{lemm:Q-computation}, we have 
\[
q(a)  = 0 \, \text{ on }\Sigma, \text{ for all }a \in \RR^{2m + 2k}.
\] 
\end{prop}
\begin{proof}
It suffices to prove that $q(a) = 0$ for all unit vectors $a \in \RR^{2m + 2k}$, so below we assume that $|a| = 1$. Combining the stability assumption, Lemma~\ref{lemm:2nd-variation-with-cutoff}, and Lemma~\ref{lemm:Q-computation}, we see that for all $f \in C^\infty_c(\Sigma)$ there holds
\begin{equation}
\begin{split}
0 \leq \ & \delta^2V(fJ_N a^\perp, fJ_N a^\perp)\\
=\ & \int_{\Sigma} \big( |\nabla f|^2 |a^\perp|^2 + f^2q(a) \big) \vol_\Sigma \leq \int_{\Sigma} \big( |\nabla f|^2 + f^2q(a) \big)\vol_\Sigma,
\end{split}
\end{equation}
where to obtain the last inequality we merely estimated $|a^\perp|$ from above by $1$. Since $\Sigma$ is complete, the fact that
\[
0 \leq \int_{\Sigma} \big( |\nabla f|^2 + f^2q(a) \big) \vol_\Sigma \, \text{ for all }f \in C^\infty_c (\Sigma)
\]
implies, by Theorem~\ref{thm:S-FC}, that there exists a positive solution $v_a$ on $\Sigma$ to 
\begin{equation}\label{eq:va-PDE}
\Delta v_a = q(a)v_a.
\end{equation}
Next, letting $w_a = \log v_a$, we compute
\begin{equation}\label{eq:wa-PDE}
\Delta w_a = \frac{\Delta v_a}{v_a} - |\nabla w_a|^2 = q(a) - |\nabla w_a|^2.
\end{equation}
Thus, for all $f \in C^\infty_c(\Sigma)$, we see that upon multiplying both sides by $f^2$, integrating by parts over $\Sigma$, and using the Cauchy-Schwarz inequality, we have
\[
\begin{split}
\int_{\Sigma} (q(a) - |\nabla w_a|^2)f^2 \vol_\Sigma =\ & \int_{\Sigma} f^2 \Delta w_a \vol_\Sigma\\
=\ & -2\int_{\Sigma} f \bangle{\nabla f, \nabla w_a} \vol_\Sigma\\
\geq\ & -\int_{\Sigma} \big( 2|\nabla f|^2 + \frac{1}{2} f^2 |\nabla w_a|^2 \big) \vol_\Sigma.
\end{split}
\]
Rearranging gives 
\begin{equation}
\int_{\Sigma} \big( 2|\nabla f|^2 + (q(a) - \frac{1}{2}|\nabla w_a|^2) f^2 \big) \vol_\Sigma \geq 0, \, \text{ for all } f\in C^\infty_c(\Sigma).
\end{equation}
At this point we take an orthonormal basis $e_1, \cdots, e_{2m + 2k}$ of $\RR^{2m + 2k}$, and sum the above inequality over $a = e_1, \cdots, e_{2m + 2k}$. From~\eqref{eq:Q-trace-free} we obtain
\[
\int_{\Sigma} \Big( (4m + 4k) |\nabla f|^2  - \frac{1}{2}\sum_{i = 1}^{2m + 2k} |\nabla w_{e_i}|^2 f^2 \Big) \vol_\Sigma \geq 0, \, \text{ for all }f \in C^\infty_c(\Sigma).
\]
By Theorem~\ref{thm:S-FC} once again, we get a positive function $h$ on $\Sigma$ such that
\begin{equation}\label{eq:h-PDE}
(4m  + 4k)\Delta h = -\Big(\frac{1}{2}\sum_{i = 1}^{2m + 2k} |\nabla w_{e_i}|^2\Big) h \leq 0.
\end{equation}
In particular, $h$ is a positive superharmonic function on $\Sigma$, and thus $h$ is constant by our parabolicity assumption. Equation~\eqref{eq:h-PDE} and the fact that $h$ is a positive constant then imply that 
\[
\sum_{i = 1}^{2m + 2k} |\nabla w_{e_i}|^2 = 0 \text{ on }\Sigma,
\]
so each $w_{e_i}$ is constant. Going back to~\eqref{eq:wa-PDE}, we infer that $q(e_i) = 0$ for $i = 1, \cdots,2m + 2k$. Since $e_{1}, \cdots, e_{2m + 2k}$ is an arbitrary orthonormal basis of $\RR^{2m + 2k}$, we conclude that $q(a) = 0$ for all unit vector $a \in \RR^{2m + 2k}$.
\end{proof}

\subsection{Holomorphicity of stable minimal surfaces} \label{sec:holomo}

Throughout this section we assume that $m = 1$. Combining Proposition~\ref{prop:Q-vanish} and Lemma~\ref{lemm:Q-surface}, we get 
\begin{equation}\label{eq:Q-condition}
\bangle{A^+_{\tau, \tau}, v}\bangle{A^-_{\tau, \tau}, v} = \bangle{A^+_{\tau, \tau}, J_N v}\bangle{A^-_{\tau, \tau}, J_N v}, 
\end{equation}
for all $x\in \Sigma$, $\tau \in T_x\Sigma$ and $v \in N_x\Sigma$.

\begin{prop}\label{prop:Apm-vanish-local}
At each point $x$ on $\Sigma$, we have $A^{+}|_x = 0$ or $A^-|_x = 0$.
\end{prop}
\begin{proof}
Fix any unit vector $\tau \in T_x \Sigma$. By Lemma~\ref{lemm:Apm-sym-tr}, we have 
\[
A^{\pm}_{\tau, J_\Sigma\tau} = A^{\pm}_{J_\Sigma\tau, \tau} = \pm J_N A^{\pm}_{\tau,\tau},
\]
and that
\[
A^{\pm}_{J_{\Sigma}\tau, J_{\Sigma}\tau} = \pm J_N A^{\pm}_{\tau, J_{\Sigma}\tau} = - A^{\pm}_{\tau, \tau}.
\]
Thus the tensors $A^{\pm}$ are completely determined, respectively, by the entries $A^{\pm}_{\tau, \tau}$. Below we write $\xi^{\pm} : = A^{\pm}_{\tau, \tau}$ for simplicity. To prove the proposition it suffices to show that at least one of $\xi^{\pm}$ vanishes. 

Polarizing~\eqref{eq:Q-condition}, we have that 
\begin{equation}\label{eq:Apm-relation-polarized}
\begin{split}
&\bangle{\xi^+, v}\bangle{\xi^-, w} + \bangle{\xi^+, w}\bangle{\xi^-, v}\\ =\ & \bangle{\xi^+, J_Nv}\bangle{\xi^-, J_Nw} + \bangle{\xi^+, J_Nw}\bangle{\xi^-, J_Nv} , \text{ for all }v, w \in N_x\Sigma.
\end{split}
\end{equation}
Setting $v = \xi^+$ and $w = \xi^-$ in the above gives
\[
|\xi^+|^2 |\xi^-|^2 + \bangle{\xi^+, \xi^-}^2 = 0 + \bangle{\xi^+, J_N\xi^-}\bangle{\xi^-, J_N\xi^+} =- \bangle{\xi^+, J_N\xi^-}^2.
\]
Therefore,
\[
|\xi^+|^2 |\xi^-|^2 + \bangle{\xi^+, \xi^-}^2  + \bangle{\xi^+, J_N\xi^-}^2 = 0.
\]
In particular, $|\xi^+||\xi^-| = 0$, and hence we have $\xi^+ = 0$ or $\xi^- = 0$ as desired.
\end{proof}

To continue, for any given $N\Sigma$-valued tensor $S$ on $\Sigma$, in analogy with~\eqref{eq:defn-Apm}, we let
\begin{equation} \label{eq:ddbar-defn}
D^{0, 1}_{v} S = \frac{1}{2}\big( D_{v}S + J_N D_{J_{\Sigma} v} S \big),\quad 
D^{1, 0}_{v} S = \frac{1}{2}\big( D_{v}S - J_N D_{J_{\Sigma} v} S \big),
\end{equation}
and observe the following relations analogous to~\eqref{eq:Apm-intertwine}:
\begin{equation}\label{eq:ddbar-intertwine}
\begin{split}
D^{0,1}_{J_{\Sigma }v} =\ & -J_N D^{0, 1}_{v},\\
D^{1, 0}_{J_{\Sigma }v} =\ & J_N D^{1, 0}_{v}.
\end{split}
\end{equation}
Also, the formal adjoints of $D^{0, 1}$ and $D^{1, 0}$, with respect to the usual $L^2$ inner product defined using $\langle \cdot, \cdot\rangle$, are given by 
\[
(D^{0, 1})^* = - e_i \hk D^{1, 0}_{e_i}, \quad (D^{1, 0})^* = -e_i \hk D^{0, 1}_{e_i},
\]
and a direct computation gives 
\begin{align}
(D^{1, 0})^*D^{1, 0} S + (D^{0, 1})^* D^{0, 1}S =\ & D^* DS,\label{eq:del-plus-delbar}\\
(D^{1, 0})^*D^{1, 0} S - (D^{0, 1})^* D^{0, 1}S = \ & \frac{1}{2}J_N \big( D^2_{e_i, J_{\Sigma}e_i}S - D^2_{J_{\Sigma}e_i, e_i} S\big), \label{eq:del-minus-delbar}
\end{align}
both of which are standard facts. As motivation for the next result, we mention that the complex structure $J_N$ makes $N\Sigma$ a complex vector bundle, and our assumption from the start that $DJ_N = 0$ implies that the connection $D$ is $\CC$-linear. Since we are now assuming that $\Sigma$ is a surface, the curvature of $D$ automatically has only the $(1, 1)$-component, and thus $D^{0, 1}$ is compatible with a holomorphic structure on $N\Sigma$. The next proposition can then be interpreted as establishing, respectively, the holomorphicity and anti-holomorphicity of $A^+$ and $A^-$.

\begin{prop}\label{prop:Apm-holo}
We have $D^{0, 1}A^+ = 0$ and $D^{1, 0}A^{-} = 0$.
\end{prop}
\begin{proof}
By~\eqref{eq:Apm-intertwine} and Lemma~\ref{lemm:Apm-sym-tr} we have 
\[
\begin{split}
A^+_{J_{\Sigma}v, w} =\ & J_NA^+_{v, w},\\
 A^+_{v, J_{\Sigma}w} =\ & A^+_{J_{\Sigma}w, v} = J_N A^+_{w, v} = J_NA^+_{v, w}.
\end{split}
\]
Given $x \in \Sigma$, choose any unit tangent vector $\tau$ at $x$. Then since $\{\tau, J_{\Sigma}\tau\}$ is an orthonormal basis for $T_x\Sigma$, the above relations, together with $D J_N = 0$, $\nabla J_{\Sigma} = 0$, and~\eqref{eq:ddbar-intertwine}, imply that to obtain $D^{0, 1}A^+ = 0$ at $x$, it suffices to prove that $(D^{0, 1}A^+)_{\tau, \tau, \tau} = 0$. For that purpose, we extend $\tau$ to a neighborhood of $x$ to be a unit vector field such that $\nabla\tau = 0$ at $x$. Then, using~\eqref{eq:ddbar-defn} and~\eqref{eq:defn-Apm}, at the point $x$ we have
\[
\begin{split}
4(D^{0, 1}A^+)_{\tau, \tau, \tau} =\ & 2 (DA^+)_{\tau, \tau, \tau} + 2J_N (DA^+)_{J_{\Sigma}\tau, \tau, \tau} \\
= \ & 2 D_{\tau}(A^+_{\tau, \tau}) + 2J_N D_{J_{\Sigma}\tau}(A^+_{\tau, \tau})\\
=\ & D_{\tau}(A_{\tau,\tau}) - J_N D_{\tau}(A_{J_{\Sigma}\tau, \tau}) + J_N D_{J_{\Sigma}\tau}(A_{\tau, \tau}) + D_{J_{\Sigma}\tau}(A_{J_{\Sigma}\tau, \tau}).
\end{split}
\]
Since the ambient space is flat, we can apply the Codazzi equation $D_{u} A_{v, w} = D_{v} A_{u, w}$ to the above to get
\[
\begin{split}
4(D^{0, 1}A^+)_{\tau, \tau, \tau} =\ & D_{\tau}(A_{\tau, \tau} + A_{J_{\Sigma}\tau, J_{\Sigma}\tau}) +  J_ND_{\tau}(A_{J_{\Sigma}\tau, \tau} - A_{J_{\Sigma}\tau, \tau}) = 0,
\end{split}
\]
since $A$ is trace-free. The proof that $D^{1, 0}A^{-} = 0$ is similar.
\end{proof}

It follows from Proposition~\ref{prop:Apm-holo} together with the identities~\eqref{eq:del-plus-delbar} and~\eqref{eq:del-minus-delbar} above that 
\begin{equation}\label{eq:A-plus-PDE}
\begin{split}
0 = \ & ((D^{0, 1})^*D^{0, 1}A^+)_{v, w}\\
 =\ &  \frac{1}{2}(D^*DA^+)_{v, w} - \frac{1}{2}J_N\big( F^D_{e_1, e_2}A^+_{v, w} - A^+_{R^{\Sigma}_{e_1, e_2}v, w} - A^+_{v, R^{\Sigma}_{e_1, e_2}w} \big),
\end{split}
\end{equation}
where $F^{D}$ is the curvature of $D$ and $R^{\Sigma}$ is the curvature of the induced metric on $\Sigma$. A similar differential equation is satisfied by $A^{-}$. As we show in Corollary~\ref{coro:UCP} below, this implies that $A^{+}$ and $A^-$ both satisfy a unique continuation principle. The result is a consequence of elementary complex analysis in view of Proposition~\ref{prop:Apm-holo} and the interpretation given before its statement, but we prefer to give a proof that avoids complex-analytic methods. We shall use the following well-known theorem in the proof of Corollary~\ref{coro:UCP}.

\begin{thm}\label{thm:UCP}
Let $U$ be an open set in a smooth Riemannian manifold of dimension $n \geq 2$. Given $N \geq 1$, let $u \colon U \to \RR^{N}$ be a smooth function satisfying for some constant $C > 0$ that 
\[
|\Delta u | \leq C(|u| + |\nabla u|) \quad \text{on }U,
\]
where $\Delta$ is the Laplace-Beltrami operator. Then, whenever $x_0 \in U$ is such that
\[
\nabla^mu|_{x_0} = 0 \, \text{ for all }m \in \NN \cup \{0\},
\]
we have that $u$ vanishes on a neighborhood of $x_0$. 
\end{thm}
\begin{proof}
This is a classical result. For a proof, see for instance~\cite{Kazdan}.
\end{proof}

\begin{coro}\label{coro:UCP}
Define the sets $\Omega^{\pm} = \{x \in \Sigma\ |\  A^{\pm}|_{x} = 0\}$ and let $\mathring{\Omega}^{\pm}$ denote their interiors. Then either $\mathring{\Omega}^{+}$ is empty, or it is all of $\Sigma$. The same statement holds for $\mathring{\Omega}^{-}$. 
\end{coro}
\begin{proof}
We only prove the statement for $\mathring{\Omega}^+$. The proof for $\mathring{\Omega}^{-}$ is essentially the same. Since $\Sigma$ is connected by assumption, and since $\mathring{\Omega}^+$ is open, to establish the asserted dichotomy it suffices to prove that $\mathring{\Omega}^{+}$ is also closed. To that end, let $x_0$ be a point in its closure, and introduce local orthonormal frames $\{\tau_{1}, \tau_{2}\}$ and $\{\nu_{1}, \cdots, \nu_{2k}\}$ for $T\Sigma$ and $N\Sigma$, respectively, on some open set $U$ containing $x_0$. Define
\[
u_{ij\alpha} = \bangle{A^+_{\tau_i, \tau_j}, \nu_{\alpha}}, \text{ for }i, j \in \{1,2\} \text{ and } \alpha \in \{1, \cdots , 2k\},
\]
and regard  $u = (u_{ij\alpha})$ as a function from $U$ into $\RR^{8k}$. Noting that $u$ vanishes on a neighborhood of each $x \in \mathring{\Omega}^+ \cap U$, and recalling that $x_0$ is the limit of a sequence of such points, we deduce from the smoothness of $u$ that 
\begin{equation}\label{eq:infinite-order-of-vanishing}
\nabla^mu|_{x_0} = 0, \text{ for all }m \in \NN \cup \{0\}.
\end{equation}
To continue, we let
\begin{equation}\label{eq:connection-coefficients}
\Gamma_{ij}^{l} = \bangle{\nabla_{\tau_i}\tau_j, \tau_l}, \quad  \theta_{i\alpha}^{\beta} = \bangle{D_{\tau_i}\nu_{\alpha}, \nu_{\beta}},
\end{equation}
and likewise regard $\Gamma$ and $\theta$ as vector-valued functions on $U$. Shrinking $U$ if necessary, we get constants $C_1$ and $C_2$ so that, pointwise on $U$, there holds
\begin{equation}\label{eq:connection-and-curvature-bounds}
|\Gamma| + |\nabla \Gamma| + |\theta| + |\nabla \theta| \leq C_1, \quad |F^{D}| + |R^{\Sigma}| \leq C_2.
\end{equation}
The Leibniz rule gives 
\begin{equation}\label{eq:A-plus-derivative-1}
\begin{split}
(D_{\tau_p}A^{+})_{\tau_i, \tau_j} =\ & \big(\tau_{p}(u_{ij\alpha}) + u_{ij\beta} \theta_{p\beta}^{\alpha} - \Gamma_{pi}^{l}u_{lj\alpha} - \Gamma_{pj}^{l}u_{il\alpha}\big)\nu_{\alpha},
\end{split}
\end{equation}
from which we get
\begin{equation}\label{eq:A-plus-bounds-1}
|DA^+| \leq C(|u| + |\nabla u|) \quad \text{on }U, 
\end{equation}
where $C$ depends only on $C_1$ and $k$. Differentiating once more in~\eqref{eq:A-plus-derivative-1} and summing over $p$ yields
\[
\begin{split}
-(D^*DA^+)_{\tau_i, \tau_j}  =\ &  D_{\tau_p}((D_{\tau_p}A^+)_{\tau_i, \tau_j}) - (D_{\nabla_{\tau_p}\tau_p}A^+)_{\tau_i, \tau_j} \\
&- \Gamma_{pi}^{l}(D_{\tau_p}A^+)_{\tau_l, \tau_j} - \Gamma_{pj}^{l}(D_{\tau_p}A^+)_{\tau_i, \tau_l}.
\end{split}
\]
Applying~\eqref{eq:A-plus-derivative-1} to the first two appearances of $DA^+$ on the right-hand side, singling out the expressions making up $(\Delta u_{ij\alpha})\nu_{\alpha}$, and then using~\eqref{eq:connection-and-curvature-bounds} and~\eqref{eq:A-plus-bounds-1} to bound the remaining terms, we find that 
\begin{equation}\label{eq:laplacian-A-plus-estimate-1}
|-(D^*DA^{+})_{\tau_i, \tau_j} - (\Delta u_{ij\alpha})\nu_{\alpha}| \leq C(|u| + |\nabla u|),
\end{equation}
where once again $C$ depends only on $C_1$ and $k$. On the other hand, from~\eqref{eq:A-plus-PDE}, which we have yet to use, together with the orthogonality of $J_{N}$ and the second bound in~\eqref{eq:connection-and-curvature-bounds}, we get
\begin{equation}\label{eq:laplacian-A-plus-estimate-2}
|(D^*D A^{+})_{\tau_i, \tau_j}| = \big| F^{D}_{\tau_1, \tau_2}A^+_{\tau_i, \tau_j} - A^+_{R^{\Sigma}_{\tau_1, \tau_2}\tau_i, \tau_j} - A^+_{\tau_i, R^{\Sigma}_{\tau_1, \tau_2}\tau_j} \big| \leq C|u|,
\end{equation}
where $C$ depends only on $C_2$ and $k$. Combining~\eqref{eq:laplacian-A-plus-estimate-2} with~\eqref{eq:laplacian-A-plus-estimate-1}, and noting that $i, j \in \{1, 2\}$ are arbitrary in both estimates, we obtain
\begin{equation}\label{eq:diff-ineq-for-UCP}
|\Delta u | \leq C(|u| + |\nabla u|) \quad \text{on }U,
\end{equation}
where $C$ depends only on $C_1$, $C_2$ and $k$. Having established~\eqref{eq:diff-ineq-for-UCP} and~\eqref{eq:infinite-order-of-vanishing}, we may now invoke Theorem~\ref{thm:UCP} to conclude that $u$, and thus $A^+$, vanishes on a neighborhood of $x_0$. That is, $x_0 \in \mathring{\Omega}^+$. This proves that $\mathring{\Omega}^+$ is closed, and concludes the proof.
\end{proof}

\begin{prop}\label{prop:Apm-vanish-global}
At least one of $A^+, A^-$ vanishes identically on $\Sigma$.
\end{prop}
\begin{proof}
By Proposition~\ref{prop:Apm-vanish-local}, for all $x \in \Sigma$ we have $A^+|_x = 0$ or $A^-|_x = 0$. In other words, in the notation of Corollary~\ref{coro:UCP}, there holds
\begin{equation}\label{eq:Sigma-union}
\Sigma = \Omega^+ \cup \Omega^-.
\end{equation}
Now, if the desired conclusion does not hold, then Corollary~\ref{coro:UCP} forces both $\Omega^+$ and $\Omega^-$ to have empty interior, and thus be nowhere dense since they are already closed by continuity. Consequently $\Omega^+ \cup \Omega^-$ is nowhere dense as well, which contradicts~\eqref{eq:Sigma-union}. Thus at least one of $A^+$ or $A^-$ vanishes on all of $\Sigma$, as asserted.
\end{proof}

We are now ready to give the proof of Theorem~\ref{thm:main}.

\begin{proof}[Proof of Theorem~\ref{thm:main}]
By Proposition~\ref{prop:Apm-vanish-global}, Proposition~\ref{prop:infinitesimally-holomorphic}, and Remark~\ref{rmk:pm-holomorphic}, at least one of $J$ or $J_{-}$ is parallel with respect to $\overline{\nabla}$ as a section of $\Endo(T\Sigma \oplus N\Sigma)$. Below we assume that $\overline{\nabla}J = 0$. The argument is exactly the same in the other case. Viewing the standard basis vectors on $\RR^{2 + 2k}$ as parallel sections of $T\RR^{2 + 2k}$ and pulling them back by the immersion $F$, we obtain a $\overline{\nabla}$-parallel global orthonormal frame $e_1, \cdots, e_{2+2k}$ for $F^*(T\RR^{2 + 2k})$. Using $\overline{\nabla}J = 0$, we have on $\Sigma$ that
\[
\begin{split}
d\bangle{e_i, Je_j} =\ & \bangle{\overline{\nabla}e_i, Je_j} + \bangle{e_i, \overline{\nabla}(Je_j)} \\
=\ &  \bangle{\overline{\nabla}e_i, Je_j} + \bangle{e_i, J(\overline{\nabla}e_j)}  = 0.
\end{split}
\]
Therefore $J_{ij}: = \bangle{e_i, Je_j}$ is constant on $\Sigma$. 
The constant $(2+2k) \times (2+2k)$-matrix with $J_{ij}$ as entries then defines a parallel orthogonal complex structure $\overline{J}$ on $\RR^{2+2k}$ which restricts to $J$ along $\Sigma$. By the construction of the latter, we conclude that $\Sigma$ is holomorphic with respect to $\overline{J}$.
\end{proof}

\section{Possible further generalizations} \label{sec:conclusion}

It is natural to ask to what extent Theorem~\ref{thm:main} can be further generalized to other interesting geometric calibrations, in the context of special Riemannian holonomy. To be more precise, consider a constant coefficient calibration $p$-form $\alpha_\circ$ on $\RR^n$. If $\Sigma^p$ is an oriented $p$-dimensional submanifold immersed in $\RR^n$, let $E = F^*(T\RR^{n}) = T\Sigma \oplus N\Sigma$. We say that $\Sigma$ admits an $\alpha_\circ$-structure if there is a smooth section $\alpha$ of $\Lambda^p ( E^*)$ such that at every point $x \in \Sigma$, the element $\alpha_x \in \Lambda^p E_x^* = \Lambda^p T_x^* \RR^n$ corresponds to the model calibration form $\alpha_\circ$ on $\RR^n$ via some isomorphism $E_x = T_x \RR^n \cong \RR^n$. Suppose further that this $\alpha_\circ$-structure on $\Sigma$ is parallel with respect to the connection $\nabla \oplus D$ on $T\Sigma \oplus N\Sigma$ induced from the Euclidean connection $\overline{\nabla}$ on $\RR^n$. If $\Sigma$ is stable, minimal, complete, and parabolic, then can we conclude that there exists a $\overline{\nabla}$-parallel calibration form $\widetilde \alpha$ on $\RR^n$ whose pullback to $\Sigma$ is $\alpha$, and such that $\Sigma$ is $\widetilde \alpha$-calibrated? We note that the hypothesis that there exists a parallel $\alpha_\circ$-structure on $\Sigma$ is necessary for such a result to hold, as it is a consequence of the conclusion.

The above question was posed in~\cite[Question 3.20]{KM-extrinsic} in the context of \emph{compliant} calibrations. A calibration on $\RR^n$ is compliant if its stabilizer in $\mathrm{SO}(n)$ satisfies a particular Lie-theoretic property. In particular, all of the geometrically interesting calibrations, including K\"ahler, special Lagrangian, associative, coassociative, and Cayley, are compliant in this sense. We remark that a notion of ``infinitesimally $\alpha$-calibrated'' is formulated in~\cite[Definition 3.16]{KM-extrinsic} for compliant calibrations, generalizing the $(J_\Sigma, J_N)$-infinitesimally holomorphic condition recalled in Section~\ref{sec:inf-holo} above. The definition is given in terms of Lie-theoretic conditions on the second fundamental form and on the holonomy algebra of the induced connection on $E$. These are precisely the conditions that are imposed on the immersion $F \colon \Sigma \to \RR^n$ if it was indeed calibrated with respect to a parallel compliant calibration form $\alpha$. (In the particular case studied here, the holonomy condition was not explicitly used but it was present.)

The first such case would be the K\"ahler calibration $\frac{1}{m!} \omega^m$ for $m > 1$; that is, a submanifold $\Sigma^{2m}$ for $m > 1$ immersed into $\RR^{2m+2k}$. Now we need to assume that not only is there a $D$-parallel orthogonal complex structure $J_N$ on $N\Sigma$, but also a $\nabla$-parallel orthogonal complex structure on $T\Sigma$, which is automatic when $m=1$. (This pair is an $\alpha_\circ$-structure on $\Sigma$ for $\alpha_0 = \frac{1}{m!} \omega^m$.) Most of the intermediate results we derived in this paper hold for any $m$, but not all. In particular Lemmas~\ref{lemm:Apm-sym-tr} and~\ref{lemm:Q-surface} required $m=1$, and we used these in the final Section~\ref{sec:holomo} of the proof. If suitable adaptations of these results for $m>1$ hold, they must necessarily be significantly more complicated. This can be partly attributed to the fact that the calibration form $\frac{1}{m!} \omega^m$ for $2m$-dimensional calibrated submanifolds in $\CC^{m+k}$ corresponds to a vector cross product if and only if $m=1$. (See, for example~\cite{CKM-bubbling, LL}.)

If such stability theorems do hold for some other interesting geometric calibrations, then perhaps those whose calibration form corresponds to a vector cross product would be the best candidates, as they are direct analogues of the case considered here and by Micallef. These would be the associative calibration on $\RR^7$ and the Cayley calibration on $\RR^8$. The authors spent a considerable amount of time trying to prove such theorems. In particular, we considered the obvious generalizations of the  quadratic form $q$ associated to our choice of ``special variations'', analogous to that of Lemma~\ref{lemm:Q-computation}, and found that they did not have trace zero. In fact, the step in the proof of Theorem~\ref{thm:main} which uses the trace-free property of $q$ would have still worked, with a little more effort, if the trace of $q$ was nonpositive, but we found it to be strictly positive. (We report these results here without proof.) This is somewhat mysterious, but it does not rule out the possibility that such theorems may still hold, albeit with proofs that are not directly analogous to the proof of our Theorem~\ref{thm:main}. It is also of course possible that the higher degree calibrations truly do have quite different behaviour. (For example, in~\cite{CKM-variational} we discovered one surprising way in which the Cayley calibration behaves quite differently from the special Lagrangian, associative, and coassociative calibrations.)

A closely related but different question is the following. Suppose that $\RR^n$, or a more general ambient Riemannian manifold $(M^n, h)$ for that matter, is equipped with a particular parallel calibration $p$-form $\alpha$ from the outset. Suppose $F \colon \Sigma^p \to (M^n, h)$ is a stable minimal immersion. Under what additional assumptions can we conclude that the immersion must be in fact $\alpha$-calibrated? A classical result in this direction is the theorem of Lawson--Simons~\cite{LS}, which shows in particular that any immersed compact stable minimal submanifold of $\CC\PP^{m}$ (with the Fubini--Study metric) must be complex up to changing orientation. A similar result was later obtained for $\HH\PP^m$ by Ohnita~\cite{Ohnita}. However, the authors are unaware of any existing results in this vein when the ambient space is a $\mathrm{G}_2$ or $\mathrm{Spin}(7)$ manifold. Note that this question is more analogous to the Bourguignon--Lawson and Siu--Yau results mentioned in Section~\ref{sec:intro}, as the additional geometric structures are already present, whereas in the Micallef result, the existence of the appropriate extra geometric structure is proved rather than assumed. 

Finally, in~\cite{CKM-bubbling} the authors initiated a study of \emph{Smith maps}, which are a natural class of maps that generalize $J$-holomorphic maps into K\"ahler manifolds to the setting of $\mathrm{G}_2$ or $\mathrm{Spin}(7)$ manifolds, possessing the correct properties with respect to conformal geometry. A Smith map $u \colon (\Sigma^p, g) \to (M^n, h, \alpha)$ is in particular a locally minimizing $p$-harmonic map. (See also~\cite{CKM-variational} and~\cite{IK} for more on Smith maps.) It is natural to ask whether a generalization of the Siu--Yau theorem~\cite{SY} to this setting might exist. That is, could it be that a stable $p$-harmonic map $u \colon (\Sigma^p, g) \to (M^n, h, \alpha)$ where $M$ satisfies some notion of ``positive bisectional curvature'' must be a Smith map? In this context, the notion of $\alpha$-sectional curvature introduced recently in~\cite{BIM} may be particularly relevant.

\addcontentsline{toc}{section}{References}

\bibliographystyle{plain}
\bibliography{CKM-stable-arXiv.bib}

\begin{bibdiv}
\begin{biblist}

\bib{BIM}{unpublished}{
      author={Broder, Kyle},
      author={Iliashenko, Anton},
      author={Madnick, Jesse},
       title={Hyperbolicity and schwarz lemmas in calibrated geometry},
         url={https://doi.org/10.48550/arXiv.2507.16313},
        note={https://doi.org/10.48550/arXiv.2507.16313},
}

\bib{BL}{article}{
      author={Bourguignon, Jean-Pierre},
      author={Lawson, H.~Blaine, Jr.},
       title={Stability and isolation phenomena for {Y}ang-{M}ills fields},
        date={1981},
        ISSN={0010-3616,1432-0916},
     journal={Comm. Math. Phys.},
      volume={79},
      number={2},
       pages={189\ndash 230},
         url={http://projecteuclid.org/euclid.cmp/1103908963},
      review={\MR{612248}},
}

\bib{Bryant}{article}{
      author={Bryant, Robert~L.},
       title={Conformal and minimal immersions of compact surfaces into the
  {$4$}-sphere},
        date={1982},
        ISSN={0022-040X,1945-743X},
     journal={J. Differential Geometry},
      volume={17},
      number={3},
       pages={455\ndash 473},
         url={http://projecteuclid.org/euclid.jdg/1214437137},
      review={\MR{679067}},
}

\bib{CKM-bubbling}{article}{
      author={Cheng, Da~Rong},
      author={Karigiannis, Spiro},
      author={Madnick, Jesse},
       title={Bubble tree convergence of conformally cross product preserving
  maps},
        date={2020},
        ISSN={1093-6106,1945-0036},
     journal={Asian J. Math.},
      volume={24},
      number={6},
       pages={903\ndash 984},
         url={https://doi.org/10.4310/AJM.2020.v24.n6.a1},
      review={\MR{4312754}},
}

\bib{CKM-variational}{article}{
      author={Cheng, Da~Rong},
      author={Karigiannis, Spiro},
      author={Madnick, Jesse},
       title={A variational characterization of calibrated submanifolds},
        date={2023},
        ISSN={0944-2669,1432-0835},
     journal={Calc. Var. Partial Differential Equations},
      volume={62},
      number={6},
       pages={Paper No. 174, 38},
         url={https://doi.org/10.1007/s00526-023-02513-7},
      review={\MR{4603874}},
}

\bib{CO}{article}{
      author={Chern, Shiing-shen},
      author={Osserman, Robert},
       title={Complete minimal surfaces in euclidean {$n$}-space},
        date={1967},
        ISSN={0021-7670},
     journal={J. Analyse Math.},
      volume={19},
       pages={15\ndash 34},
         url={https://doi.org/10.1007/BF02788707},
      review={\MR{226514}},
}

\bib{EL}{book}{
      author={Eells, James},
      author={Lemaire, Luc},
       title={Selected topics in harmonic maps},
      series={CBMS Regional Conference Series in Mathematics},
   publisher={Conference Board of the Mathematical Sciences, Washington, DC; by
  the American Mathematical Society, Providence, RI},
        date={1983},
      volume={50},
        ISBN={0-8218-0700-5},
         url={https://doi.org/10.1090/cbms/050},
      review={\MR{703510}},
}

\bib{S-FC}{article}{
      author={Fischer-Colbrie, Doris},
      author={Schoen, Richard},
       title={The structure of complete stable minimal surfaces in
  {$3$}-manifolds of nonnegative scalar curvature},
        date={1980},
        ISSN={0010-3640,1097-0312},
     journal={Comm. Pure Appl. Math.},
      volume={33},
      number={2},
       pages={199\ndash 211},
         url={https://doi.org/10.1002/cpa.3160330206},
      review={\MR{562550}},
}

\bib{FS}{unpublished}{
      author={Fraser, Ailana},
      author={Schoen, Richard},
       title={Stability and largeness properties of minimal surfaces in higher
  codimension},
         url={https://doi.org/10.48550/arXiv.2303.07423},
        note={J. Eur. Math. Soc., to appear, arXiv.2303.07423},
}

\bib{IK}{article}{
      author={Iliashenko, Anton},
      author={Karigiannis, Spiro},
       title={A special class of {$k$}-harmonic maps inducing calibrated
  fibrations},
        date={2025},
        ISSN={1073-2780,1945-001X},
     journal={Math. Res. Lett.},
      volume={32},
      number={5},
       pages={1481\ndash 1519},
         url={https://doi.org/10.4310/mrl.251215003006},
      review={\MR{5009217}},
}

\bib{Kazdan}{article}{
      author={Kazdan, Jerry~L.},
       title={Unique continuation in geometry},
        date={1988},
        ISSN={0010-3640},
     journal={Comm. Pure Appl. Math.},
      volume={41},
      number={5},
       pages={667\ndash 681},
         url={https://doi.org/10.1002/cpa.3160410508},
      review={\MR{948075}},
}

\bib{KM-extrinsic}{article}{
      author={Karigiannis, Spiro},
      author={Mart\'in-Merch\'an, Luc\'ia},
       title={Extrinsic geometry of calibrated submanifolds},
        date={2024},
        ISSN={0025-5874,1432-1823},
     journal={Math. Z.},
      volume={307},
      number={2},
       pages={Paper No. 33, 26},
         url={https://doi.org/10.1007/s00209-024-03503-x},
      review={\MR{4745761}},
}

\bib{LL}{article}{
      author={Lee, Jae-Hyouk},
      author={Leung, Naichung~Conan},
       title={Instantons and branes in manifolds with vector cross products},
        date={2008},
        ISSN={1093-6106,1945-0036},
     journal={Asian J. Math.},
      volume={12},
      number={1},
       pages={121\ndash 143},
         url={https://doi.org/10.4310/AJM.2008.v12.n1.a9},
      review={\MR{2415016}},
}

\bib{LS}{article}{
      author={Lawson, H.~Blaine, Jr.},
      author={Simons, James},
       title={On stable currents and their application to global problems in
  real and complex geometry},
        date={1973},
        ISSN={0003-486X},
     journal={Ann. of Math. (2)},
      volume={98},
       pages={427\ndash 450},
         url={https://doi.org/10.2307/1970913},
      review={\MR{324529}},
}

\bib{Micallef}{article}{
      author={Micallef, Mario~J.},
       title={Stable minimal surfaces in {E}uclidean space},
        date={1984},
        ISSN={0022-040X,1945-743X},
     journal={J. Differential Geom.},
      volume={19},
      number={1},
       pages={57\ndash 84},
         url={http://projecteuclid.org/euclid.jdg/1214438423},
      review={\MR{739782}},
}

\bib{MW}{article}{
      author={Micallef, Mario~J.},
      author={Wolfson, Jon~G.},
       title={The second variation of area of minimal surfaces in
  four-manifolds},
        date={1993},
        ISSN={0025-5831,1432-1807},
     journal={Math. Ann.},
      volume={295},
      number={2},
       pages={245\ndash 267},
         url={https://doi.org/10.1007/BF01444887},
      review={\MR{1202392}},
}

\bib{Ohnita}{article}{
      author={Ohnita, Yoshihiro},
       title={Stable minimal submanifolds in compact rank one symmetric
  spaces},
        date={1986},
        ISSN={0040-8735},
     journal={Tohoku Math. J. (2)},
      volume={38},
      number={2},
       pages={199\ndash 217},
         url={https://doi.org/10.2748/tmj/1178228488},
      review={\MR{843807}},
}

\bib{SY}{article}{
      author={Siu, Yum~Tong},
      author={Yau, Shing~Tung},
       title={Compact {K}\"ahler manifolds of positive bisectional curvature},
        date={1980},
        ISSN={0020-9910,1432-1297},
     journal={Invent. Math.},
      volume={59},
      number={2},
       pages={189\ndash 204},
         url={https://doi.org/10.1007/BF01390043},
      review={\MR{577360}},
}

\end{biblist}
\end{bibdiv}

\end{document}